\documentclass[12pt]{article}%
\usepackage{graphicx}
\usepackage{amsmath}
\usepackage{amsfonts}
\usepackage{amssymb}%
\providecommand{\U}[1]{\protect\rule{.1in}{.1in}}
\providecommand{\U}[1]{\protect\rule{.1in}{.1in}}
\providecommand{\U}[1]{\protect\rule{.1in}{.1in}}
\newtheorem{theorem}{Theorem}

\newtheorem{algorithm}[theorem]{Algorithm}

\numberwithin{equation}{section}
\numberwithin{theorem}{section}
\newtheorem{definition}[theorem]{Definition}

\newtheorem{lemma}[theorem]{Lemma}

\newtheorem{remark}[theorem]{Remark}

\newenvironment{proof}[1][Proof]{\textbf{#1.} }{\ \rule{0.5em}{0.5em}}
\begin{document}

\title{\textbf{The Split Common Null Point Problem}}
\author{Charles Byrne$^{1}$, Yair Censor$^{2}$, Aviv Gibali$^{3}$ and Simeon
Reich$^{3}$\bigskip\\$^{1}$Department of Mathematical Sciences\\University of Massachusetts at Lowell\\Lowell, MA 01854, USA\\(Charles\_Byrne@uml.edu)\bigskip\\$^{2}$Department of Mathematics, University of Haifa\\Mt.\ Carmel, 31905 Haifa, Israel\\(yair@math.haifa.ac.il)\bigskip\\$^{3}$Department of Mathematics\\The Technion - Israel Institute of Technology\\32000 Haifa, Israel\\(avivg@techunix.technion.ac.il, sreich@techunix.technion.ac.il)\bigskip}
\date{April 15, 2012}
\maketitle

\begin{abstract}
We introduce and study the Split Common Null Point Problem (SCNPP) for
set-valued maximal monotone mappings in Hilbert spaces. This problem
generalizes our Split Variational Inequality Problem (SVIP) [Y. Censor, A.
Gibali and S. Reich, Algorithms for the split variational inequality problem,
Numerical Algorithms 59 (2012), 301--323]. The SCNPP with only two set-valued
mappings entails finding a zero of a maximal monotone mapping in one space,
the image of which under a given bounded linear transformation is a zero of
another maximal monotone mapping. We present four iterative algorithms that
solve such problems in Hilbert spaces, and establish weak convergence for one
and strong convergence for the other three.

\end{abstract}


\section{Introduction}

In this paper we introduce and study the \textit{Split Common Null Point
Problem} for set-valued mappings in Hilbert spaces. Let $\mathcal{H}_{1}$ and
$\mathcal{H}_{2}$ be two real Hilbert spaces. Given set-valued mappings
$B_{i}:\mathcal{H}_{1}\rightarrow2^{\mathcal{H}_{1}}$, $1\leq i\leq p$, and
$F_{j}:\mathcal{H}_{2}\rightarrow2^{\mathcal{H}_{2}}$, $1\leq j\leq r,$
respectively, and bounded linear operators $A_{j}:\mathcal{H}_{1}%
\rightarrow\mathcal{H}_{2}$, $1\leq j\leq r$, the problem is formulated as
follows:%
\begin{gather}
\text{find a point }x^{\ast}\in\mathcal{H}_{1}\text{ such that }0\in\cap
_{i=1}^{p}B_{i}(x^{\ast})\label{eq:mszp-1}\\
\text{and such that the points}\nonumber\\
y_{j}^{\ast}=A_{j}\left(  x^{\ast}\right)  \in\mathcal{H}_{2}\text{ solve
}0\in\cap_{j=1}^{r}F_{j}(y_{j}^{\ast}). \label{eq:mszp-2}%
\end{gather}
We denote this problem by SCNPP$(p,r)$ to emphasize the multiplicity of
mappings. To motivate this new problem and to understand its relationship with
other problems, we first look at the prototypical \textit{Split Inverse
Problem}\textbf{ }formulated in \cite[Section 2]{cgr}. It concerns a model in
which there are given two vector spaces $X$ and $Y$ and a linear operator
$A:X\rightarrow Y.$ In addition, two inverse problems are involved. The first
one, denoted by IP$_{1},$ is formulated in the space $X$ and the second one,
denoted by IP$_{2}$, is formulated in the space $Y.$ Given these data, the
\textit{Split Inverse Problem} (SIP) is formulated as follows:%
\begin{gather}
\text{find a point }x^{\ast}\in X\text{ that solves IP}_{1}\text{ }\\
\text{and such that}\nonumber\\
\text{the point }y^{\ast}=A\left(  x^{\ast}\right)  \in Y\text{ solves IP}%
_{2}\text{.}%
\end{gather}

Real-world inverse problems can be cast into this framework by making
different choices of the spaces $X$ and $Y$ (including the case $X=Y$), and by
choosing appropriate inverse problems for IP$_{1}$ and IP$_{2}$. The
\textit{Split Convex Feasibility Problem} (SCFP) \cite{CE} is the first
instance of an SIP.\ The two problems IP$_{1}$ and IP$_{2}$ there are of the
\textit{Convex Feasibility Problem} (CFP) type. This formulation was used for
solving an inverse problem in radiation therapy treatment planning
\cite{CEKB,CBMT}. The SCFP has been well studied during the last two decades
both theoretically and practically; see, e.g., \cite{byrne02, CEKB} and the
references therein. Two leading candidates for IP$_{1}$ and IP$_{2}$ are the
mathematical models of the CFP and problems of constrained optimization. In
particular, the CFP formalism is in itself at the core of the modeling of many
inverse problems in various areas of mathematics and the physical sciences;
see, e.g., \cite{cap88} and references therein for an early example. Over the
past four decades, the CFP has been used to model significant real-world
inverse problems in sensor networks, radiation therapy treatment planning,
resolution enhancement and in many other areas; see \cite{cccdh11} for exact
references to all of the above. More work on the CFP can be found in
\cite{bb96,Byrne, Byrne04, cdh10}.

It is therefore natural to ask whether other inverse problems can be used for
IP$_{1}$ and IP$_{2}$, besides the CFP, and be embedded in the SIP
methodology. For example, can IP$_{1}$ = CFP in the space $X$ and can a
constrained optimization problem be IP$_{2}$ in the space $Y$? In our recent
paper \cite{cgr}\textit{ }we have made a step in this direction by formulating
an SIP with a \textit{Variational Inequality Problem} (VIP) in each of the two
spaces of the SIP, reaching a \textit{Split Variational Inequality Problem}
(SVIP). In the present paper we study an SIP with a \textit{Null Point
Problem} in each of the two spaces. As we explain below, this formulation
includes the earlier formulation with VIPs and all its special cases such as
the CFP and constrained optimization problems.

\subsection{Relations with previous work and the contribution of the present
paper}

To further motivate our study, let us look at the various problem formulations
from the point of view of their structure only, without reference to the
various assumptions made in order to prove results regarding these problems.
We put the SCNPP$(p,r)$ in the context of other SIPs and related works. First
recall the \textit{Split Variational Inequality Problem} (SVIP), which is an
SIP with a VIP in each one of the two spaces \cite{cgr}. Let $\mathcal{H}_{1}$
and $\mathcal{H}_{2}$ be two real Hilbert spaces, and assume that there are
given two operators $f:\mathcal{H}_{1}\rightarrow\mathcal{H}_{1}$ and
$g:\mathcal{H}_{2}\rightarrow\mathcal{H}_{2},$ a bounded linear operator
$A:\mathcal{H}_{1}\rightarrow\mathcal{H}_{2}$, and nonempty, closed and convex
subsets $C\subset\mathcal{H}_{1}$ and $Q\subset\mathcal{H}_{2}.$ The SVIP is
then formulated as follows:%
\begin{gather}
\text{find a point }x^{\ast}\in C\text{ such that }\left\langle f(x^{\ast
}),x-x^{\ast}\right\rangle \geq0\text{ for all }x\in C\label{eq:vip}\\
\text{and such that}\nonumber\\
\text{the point }y^{\ast}=A\left(  x^{\ast}\right)  \in Q\text{ and solves
}\left\langle g(y^{\ast}),y-y^{\ast}\right\rangle \geq0\text{ for all }y\in Q.
\label{eq:svip}%
\end{gather}
This can be structurally considered a special case of SCNPP$(1,1).$ Denoting
by \textrm{SOL}$\left(  f,C\right)  $ and \textrm{SOL}$\left(  g,Q\right)  $
the solution sets of the VIPs in (\ref{eq:vip}) and (\ref{eq:svip}),
respectively, we can also write the SVIP in the following way:%
\begin{equation}
\text{find a point }x^{\ast}\in\mathrm{SOL}\left(  f,C\right)  \text{ such
that }A\left(  x^{\ast}\right)  \in\mathrm{SOL}\left(  g,Q\right)  .
\end{equation}
Taking in (\ref{eq:vip})--(\ref{eq:svip}) $C=\mathcal{H}_{1}$, $Q=\mathcal{H}%
_{2},$ and choosing $x:=x^{\ast}-f(x^{\ast})\in\mathcal{H}_{1}$ in
(\ref{eq:vip}) and $y=A\left(  x^{\ast}\right)  -g(A\left(  x^{\ast}\right)
)\in\mathcal{H}_{2}$ in (\ref{eq:svip}), we obtain the \textit{Split Zeros
Problem} (SZP) for two operators $f:\mathcal{H}_{1}\rightarrow\mathcal{H}_{1}$
and $g:\mathcal{H}_{2}\rightarrow\mathcal{H}_{2}$, which we introduced in
\cite[Subsection 7.3]{cgr}. It is formulated as follows:%
\begin{equation}
\text{find a point }x^{\ast}\in\mathcal{H}_{1}\text{ such that }f(x^{\ast
})=0\text{ and }g(A\left(  x^{\ast}\right)  )=0. \label{eq:szp}%
\end{equation}
An important observation that should be made at this point is that if we
denote by $N_{C}\left(  v\right)  $ the \textit{normal cone} of some nonempty,
closed and convex set $C$ at a point $v\in C$, i.e.,%
\begin{equation}
N_{C}\left(  v\right)  :=\{d\in\mathcal{H}\mid\left\langle d,y-v\right\rangle
\leq0\text{ for all }y\in C\}, \label{eq:normal_cone}%
\end{equation}
and define the set-valued mapping $B$ by%
\begin{equation}
B(v):=\left\{
\begin{array}
[c]{cc}%
f(v)+N_{C}\left(  v\right)  , & v\in C,\\
\emptyset, & \text{otherwise,}%
\end{array}
\right.
\end{equation}
where $f$ is some given operator, then, under a certain continuity assumption
on $f$, Rockafellar in \cite[Theorem 3]{Rockafellar} showed that $B$ is a
maximal monotone mapping and $B^{-1}\left(  0\right)  =$ \textrm{SOL}$\left(
f,C\right)  $.

Following this idea, Moudafi \cite{Moudafi} introduced the \textit{Split
Monotone Variational Inclusion} (SMVI) which generalized the SVIP of
\cite{cgr}. Given two operators $f:\mathcal{H}_{1}\rightarrow\mathcal{H}_{1}$
and $g:\mathcal{H}_{2}\rightarrow\mathcal{H}_{2}$, a bounded linear operator
$A:\mathcal{H}_{1}\rightarrow\mathcal{H}_{2}$, and two set-valued mappings
$B_{1}:\mathcal{H}_{1}\rightarrow2^{\mathcal{H}_{1}}$ and $B_{2}%
:\mathcal{H}_{2}\rightarrow2^{\mathcal{H}_{2}}$, the SMVI is formulated as
follows:%
\begin{gather}
\text{find a point }x^{\ast}\in\mathcal{H}_{1}\text{ such that }0\in
f(x^{\ast})+B_{1}(x^{\ast})\label{eq:mvi-1}\\
\text{and such that the point}\nonumber\\
y^{\ast}=A\left(  x^{\ast}\right)  \in\mathcal{H}_{2}\text{ solves }0\in
g(y^{\ast})+B_{2}(y^{\ast})\text{.} \label{eq:mvi-2}%
\end{gather}

With the aid of simple substitutions it is clear that, structurally, SMVI is
identical with SCNPP$(1,1)$ (use only two set-valued mappings, i.e., $p=r=1,$
and put in (\ref{eq:mvi-1})--(\ref{eq:mvi-2}) above, $f=g=0$). The
applications presented in \cite{Moudafi} only deal with this situation.

Masad and Reich \cite{Masad+Reich} studied the \textit{Constrained
Multiple-Set Split Convex Feasibility Problem }(CMSSCFP). Let $r$ and $p$ be
two natural numbers. Let $C_{i},$ $1\leq i\leq p,$ and $Q_{j},$ $1\leq j\leq
r,$ be closed and convex subsets of $\mathcal{H}_{1}$ and $\mathcal{H}_{2},$
respectively; further, for each $1\leq j\leq r,$ let $A_{j}:\mathcal{H}%
_{1}\rightarrow\mathcal{H}_{2}$ be a bounded linear operator. Finally, let
$\Omega$ be another closed and convex subset of $\mathcal{H}_{1}$. The CMSSCFP
is formulated as follows:%
\begin{gather}
\text{find a point }x^{\ast}\in\Omega\label{eq:CMSSCFP-1}\\
\text{such that}\nonumber\\
x^{\ast}\in\cap_{i=1}^{p}C_{i}\text{ and }A_{j}\left(  x^{\ast}\right)  \in
Q_{j}\text{ for each }j=1,2,\ldots,r. \label{eq:CMSSCFP-2}%
\end{gather}

This is also structurally a special case of SCNPP$(p,r).$ Another related
split problem is the \textit{Split Common Fixed Point Problem} (SCFPP), first
introduced in Euclidean spaces in \cite{CS09} and later studied by Moudafi
\cite{Moudafi10} in Hilbert spaces. Given operators $U_{i}:\mathcal{H}%
_{1}\rightarrow\mathcal{H}_{1}$, $i=1,2,\ldots,p,$ and $T_{j}:\mathcal{H}%
_{2}\rightarrow\mathcal{H}_{2},$ $j=1,2,\ldots,r,$ with nonempty fixed points
sets $C_{i},$ $i=1,2,\ldots,p,$ and $Q_{j},$ $j=1,2,\ldots,r,$ respectively,
and a bounded linear operator $A:\mathcal{H}_{1}\rightarrow\mathcal{H}_{2},$
the SCFPP is formulated as follows:%
\begin{equation}
\text{find a point }x^{\ast}\in C:=\cap_{i=1}^{p}C_{i}\text{ such that
}A\left(  x^{\ast}\right)  \in Q:=\cap_{j=1}^{r}Q_{j}. \label{eq:1.15}%
\end{equation}
This is also structurally a special case of SCNPP$(p,r).$

Motivated by the CMSSCFP of \cite{Masad+Reich}, see (\ref{eq:CMSSCFP-1}%
)--(\ref{eq:CMSSCFP-2}) above, the purpose of the present paper is to
introduce the SCNPP$(p,r)$ and present algorithms for solving it. Following
\cite{Masad+Reich}, \cite{halpern} and \cite{Haugazeau68}, we are able to
establish strong convergence of three of the algorithms that we propose. These
strongly convergent algorithms can be easily adapted to the SMVI and to other
special cases of the SCNPP$(p,r)$.

Our paper is organized as follows. In Section \ref{Sec: Pre} we list several
known facts regarding operators and set-valued mappings that are needed in the
sequel. In Section \ref{Sec: Weak} we present an algorithm for solving the
SCNPP$(p,r)$ and obtain its weak convergence. In Section \ref{Sec: Str} we
propose three additional algorithms for solving the SCNPP$(p,r)$ and present
strong convergence theorems for them. Some further comments are presented in
Section \ref{sec:comments}.

\section{Preliminaries\label{Sec: Pre}}

Let $\mathcal{H}$ be a real Hilbert space with inner product $\langle
\cdot,\cdot\rangle$ and induced norm $\Vert\cdot\Vert,$ and let $D\subset
\mathcal{H}$ be a nonempty, closed and convex subset of it. We write either
$x^{k}\rightharpoonup x$ or $x^{k}\rightarrow x$ to indicate that the sequence
$\left\{  x^{k}\right\}  _{k=0}^{\infty}$ converges either weakly or strongly,
respectively, to $x$. Next we present several properties of operators and
set-valued mappings which will be useful later on. For more details on many of
the notions and results quoted here see, e.g., the recent books
\cite{BC-book-2011,burachik-book-2008}.

\begin{definition}
Let $\mathcal{H}$ be a real Hilbert space. Let $D\subset\mathcal{H}$ be a
subset of $\mathcal{H}$ and $h:D\rightarrow\mathcal{H}$ be an operator from
$D$ to $\mathcal{H}.$

\begin{enumerate}
\item $h$ is called $\nu$\texttt{-inverse strongly monotone} ($\nu$-ism) on
$D$ if there exists a number $\nu>0$ such that%
\begin{equation}
\langle h(x)-h(y),x-y\rangle\geq\nu\Vert h(x)-h(y)\Vert^{2}\text{ for all
}x,y\in D.
\end{equation}

\item $h$ is called \texttt{firmly nonexpansive} on $D$ if%
\begin{equation}
\left\langle h(x)-h(y),x-y\right\rangle \geq\left\Vert h(x)-h(y)\right\Vert
^{2}\text{\ for all }x,y\in D,
\end{equation}
i.e., if it is $1$-ism.

\item $h$ is called \texttt{Lipschitz continuous} with constant $\kappa>0$ on
$D$ if%
\begin{equation}
\Vert h(x)-h(y)\Vert\leq\kappa\Vert x-y\Vert\text{\ for all\ }x,y\in D.
\end{equation}

\item $h$ is called \texttt{nonexpansive} on $D$ if%
\begin{equation}
\left\Vert h(x)-h(y)\right\Vert \leq\left\Vert x-y\right\Vert \text{\ for all
}x,y\in D,
\end{equation}
i.e., if it is $1$-Lipschitz.

\item $h$ is called a \texttt{strict contraction} if it is Lipschitz
continuous with constant $\kappa<1$.

\item $h$ is called \texttt{hemicontinuous} if it is continuous along each
line segment in $D$.

\item $h$ is called \texttt{asymptotically regular} at $x\in D$
\cite{Browder+Petryshyn} if%
\begin{equation}
\lim_{k\rightarrow\infty}(h^{k}(x)-h^{k+1}(x))=0\text{ for all }%
x\in\mathcal{H},
\end{equation}
where $h^{k}$ denotes the k-th iterate of $h$.

\item $h$ is called \texttt{demiclosed} at $y\in\mathcal{H}$ if for any
sequence $\left\{  x^{k}\right\}  _{k=0}^{\infty}\subset D$ such that
$x^{k}\rightharpoonup\overline{x}\in D$ and $h(x^{k})\rightarrow y$, we have
$h(\overline{x})=y$.

\item $h$ is called \texttt{averaged} \cite{bbr} if there exists a
nonexpansive operator $N:D\rightarrow\mathcal{H}\ $and a number $c\in(0,1)$
such that%
\begin{equation}
h=(1-c)I+cN,
\end{equation}
where $I$ is the identity operator. In this case we also say that $h$ is
$c$-av \cite{Byrne04}.

\item $h$ is called \texttt{odd} if $D$ is symmetric, i.e., $D=-D$, and if%
\begin{equation}
h(-x)=-h(x)\text{ for all }x\in D.
\end{equation}

\end{enumerate}
\end{definition}

\begin{remark}
\label{remark:av}(i) It can be verified that if $h$ is $\nu$-ism, then it is
Lipschitz continuous with constant $\kappa=1/\nu$.

(ii) It is known that an operator $h$ is averaged if and only if its
complement $I-h$ is $\nu$-ism for some $\nu>1/2$; see, e.g., \cite[Lemma
2.1]{Byrne04}.

(iii) The operator $h$ is firmly nonexpansive if and only if its complement
$I-h$ is firmly nonexpansive. The operator $h$ is firmly nonexpansive if and
only if $h$ is $(1/2)$-av (see \cite[Proposition 11.2]{Goebel} and \cite[Lemma
2.3]{Byrne04}).

(iv) If $h_{1}$ and $h_{2}$ are $c_{1}$-av and $c_{2}$-av, respectively, then
their composition $S=h_{1}h_{2}$ is $(c_{1}+c_{2}-c_{1}c_{2})$-av. See
\cite[Lemma 2.2]{Byrne04}.
\end{remark}

\begin{definition}
Let $\mathcal{H}$ be a real Hilbert space. Let $B:\mathcal{H}\rightarrow
2^{\mathcal{H}}$ and $\lambda>0.$

(i) $B$ is called \texttt{odd} if%
\begin{equation}
B(-x)=-B(x)\text{ for all }x\in\mathcal{H}.
\end{equation}

(ii) $B$ is called a \texttt{maximal monotone mapping} if $B$ is
\texttt{monotone}, i.e.,%
\begin{equation}
\left\langle u-v,x-y\right\rangle \geq0\text{ for all }u\in B(x)\text{ and
}v\in B(y),
\end{equation}
and the \texttt{graph} $G(B)$ of $B,$%
\begin{equation}
G(B):=\left\{  \left(  x,u\right)  \in\mathcal{H}\times\mathcal{H}\mid u\in
B(x)\right\}  ,
\end{equation}
is not properly contained in the graph of any other monotone mapping.

(iii) The \texttt{domain} of $B$ is%
\begin{equation}
\operatorname*{dom}(B):=\left\{  x\in\mathcal{H}\mid B(x)\neq\emptyset
\right\}  .
\end{equation}

(iv) The \texttt{resolvent} of $B$ with parameter $\lambda$ is denoted and
defined by $J_{\lambda}^{B}:=\left(  I+\lambda B\right)  ^{-1}$, where $I$ is
the identity operator.
\end{definition}

\begin{remark}
\label{remark:resolvent} It is well known that for $\lambda>0$,

(i) $B$ is monotone if and only if the resolvent $J_{\lambda}^{B}$ of $B$ is
single-valued and firmly nonexpansive.

(ii) $B$ is maximal monotone if and only if $J_{\lambda}^{B}$ is
single-valued, firmly nonexpansive and $\operatorname*{dom}(J_{\lambda}%
^{B})=\mathcal{H}$.

(iii) The following equivalence holds:%
\begin{equation}
0\in B(x^{\ast})\Leftrightarrow x^{\ast}\in\operatorname{Fix}(J_{\lambda}%
^{B}). \label{eq:res-fix}%
\end{equation}

\end{remark}

It follows from (\ref{eq:res-fix}) that the SCNPP$(p,r)$ with two set-valued
maximal monotone mappings ($p=r=1$) can be seen as an SCFPP with respect to
their resolvents. In addition, Moudafi's SMVI can also be considered an SCFPP
with respect to $J_{\lambda}^{B_{1}}(I-\lambda f)$ and $J_{\lambda}^{B_{2}%
}(I-\lambda g)$ \cite[Fact 1]{Moudafi}. Now we present another known result;
see, e.g., \cite[Fact 2]{Moudafi}.

\begin{remark}
\label{remark:resolvent-2} Let $\mathcal{H}$ be a real Hilbert space, and let
a maximal monotone mapping $B:\mathcal{H}\rightarrow2^{\mathcal{H}}$ and an
$\alpha$-ism operator $h:\mathcal{H\rightarrow H}$ be given. Then the operator
$J_{\lambda}^{B}(I-\lambda h)$ is averaged for each $\lambda\in(0,2\alpha)$.
\end{remark}

Next we present an important class of operators, the $\mathfrak{T}$-class
operators. This class was introduced and investigated by Bauschke and
Combettes in \cite[Definition 2.2]{BC01} and by Combettes in
\cite{Combettes01}. Operators in this class were named \textit{directed
operators }by Zaknoon \cite{zaknoon} and further employed under this name by
Segal \cite{Seg08}, and by Censor and Segal \cite{cs08, CS09}. Cegielski
\cite[Def. 2.1]{Ceg08} studied these operators under the name
\textit{separating operators}. Since both \textit{directed }%
and\textit{\ separating }are keywords of other, widely-used, mathematical
entities, Cegielski and Censor have recently introduced the term
\textit{cutter operators} \cite{cc11}. This class coincides with the class
$\mathcal{F}^{\nu}$ for $\nu=1$ \cite{Crombez} and with the class
DC$_{\boldsymbol{p}}$ for $\boldsymbol{p}=-1$ \cite{Maruster and Popirlan}.
The term \textit{firmly quasi-nonexpansive} (FQNE) for $\mathfrak{T}$-class
operators was used by Yamada and Ogura \cite[Section B]{Y2004a, Yamada}
because every \textit{firmly nonexpansive} (FNE) mapping \cite[page
42]{Goebel} is obviously FQNE.

\begin{definition}
Let $\mathcal{H}$ be a real Hilbert space. An operator
$h:\mathcal{H\rightarrow H}$ is called a \texttt{cutter operator} if
$\operatorname*{dom}(h)=\mathcal{H}$ and%
\begin{equation}
\left\langle h\left(  x\right)  -x,h\left(  x\right)  -q\right\rangle
\leq0\text{\ for all }(x,q)\in\mathcal{H}\times\operatorname*{Fix}(h),
\label{def directed}%
\end{equation}
where the fixed point set $\operatorname*{Fix}(h)$ of $h$ is defined by%
\begin{equation}
\operatorname*{Fix}(h):=\{x\in\mathcal{H}\mid h(x)=x\}.
\end{equation}

\end{definition}

It can be seen that this class of operators coincides with the class of
\textit{firmly quasi-nonexpansive operators }(FQNE), which satisfy the
inequality%
\begin{equation}
\left\Vert h(x)-q\right\Vert ^{2}\leq\left\Vert x-q\right\Vert ^{2}-\left\Vert
x-h(x)\right\Vert ^{2}\text{\ for all }(x,q)\in\mathcal{H}\times
\operatorname*{Fix}(h).
\end{equation}
Note that the $\mathfrak{T}$-class operators include, among others, orthogonal
projections, subgradient projectors, resolvents of maximal monotone mappings,
and firmly nonexpansive operators. This last class was first introduced by
Browder \cite[Definition 6]{Browder67} under the name \textit{firmly
contractive operators}. Every $\mathfrak{T}$-class operator belongs to the
class $\mathcal{F}^{0}$ of operators, defined by Crombez \cite[p.
161]{Crombez}:%
\begin{equation}
\mathcal{F}^{0}:=\left\{  h:\mathcal{H\rightarrow H}\mid\left\Vert
h(x)-q\right\Vert \leq\left\Vert x-q\right\Vert \text{ for all }%
(x,q)\in\mathcal{H}\times\operatorname*{Fix}(h)\right\}  . \label{PC}%
\end{equation}
The elements of $\mathcal{F}^{0}$ are called \textit{quasi-nonexpansive} or
\textit{paracontracting operators}. A more general class of operators is the
class of \textit{demicontractive operators} (see, e.g., \cite{Maruster and
Popirlan}).

\begin{definition}
Let $\mathcal{H}$ be a real Hilbert space and let $h:\mathcal{H}%
\rightarrow\mathcal{H}$ be an operator.

(i) $h$ is called a \texttt{demicontractive operator} if there exists a number
$\beta\in\lbrack0,1)$ such that%
\begin{equation}
\left\Vert h(x)-q\right\Vert ^{2}\leq\left\Vert x-q\right\Vert ^{2}%
+\beta\left\Vert x-h(x)\right\Vert ^{2}\text{ for all }(x,q)\in\mathcal{H}%
\times\operatorname*{Fix}(h).
\end{equation}
This is equivalent to%
\begin{equation}
\left\langle x-h(x),x-q\right\rangle \geq\frac{1-\beta}{2}\left\Vert
x-h(x)\right\Vert ^{2}\text{ for all }(x,q)\in\mathcal{H}\times
\operatorname*{Fix}(h).
\end{equation}

\end{definition}

Another useful observation, already hinted to above, is that if $h:\mathcal{H}%
\rightarrow\mathcal{H}$ is monotone and hemicontinuous on a nonempty, closed
and convex subset $D$, then the set-valued mapping%
\begin{equation}
M(v)=\left\{
\begin{array}
[c]{cc}%
h(v)+N_{D}\left(  v\right)  , & v\in D,\\
\emptyset, & \text{otherwise,}%
\end{array}
\right.  \label{eq:T-mm}%
\end{equation}
is, by \cite[Theorem 3]{Rockafellar}, maximal monotone and $M^{-1}\left(
0\right)  =$ \textrm{SOL}$\left(  h,D\right)  $. Therefore, as mentioned in
\cite{Moudafi}, if we choose $B_{1}=N_{C}$ and $B_{2}=N_{Q}$ in
(\ref{eq:mvi-1}) and (\ref{eq:mvi-2}), respectively, then we get the SVIP of
(\ref{eq:vip})--(\ref{eq:svip}). Of course, this assertion also holds for our
SCNPP$(p,r)$ with two set-valued maximal monotone mappings ($p=r=1$) when we
take $B_{1}$ and $F_{1}$ to be similar to $M$ in (\ref{eq:T-mm}). This enables
us to solve the SVIP for monotone and hemicontinuous operators (which
constitute a larger class than the class of inverse strongly monotone
operators) by using our convergence theorem for the SVIP \cite[Theorem
6.3]{cgr}. In \cite[Theorem 6.3]{cgr} we also assumed \cite[Equation
(5.9)]{cgr} that for all $x^{\ast}\in\mathrm{SOL}\left(  f,C\right)  $,%
\begin{equation}
\langle f(x),P_{C}(I-\lambda f)(x)-x^{\ast}\rangle\geq0\text{ for all\ }%
x\in\mathcal{H}_{1},
\end{equation}
an assumption which is not needed for the convergence theorems we establish in
the present paper.

The next lemma is the well-known \textit{Demiclosedness Principle}
\cite{Browder}.

\begin{lemma}
\label{lem:demiclose} Let $\mathcal{H}$ be a Hilbert space, $D$ a closed and
convex subset of $\mathcal{H},$ and let $h:D\rightarrow\mathcal{H}$ be a
nonexpansive operator. Then $I-h$ is demiclosed at any $y\in\mathcal{H}$.
\end{lemma}

The next definition is due to Clarkson \cite{Clarkson}.

\begin{definition}
A Banach space $\mathcal{B}$ is said to be \texttt{uniformly convex} if to
each $\varepsilon\in(0,2],$ there corresponds a positive number $\delta
(\varepsilon)$ such that the conditions $\left\Vert x\right\Vert =\left\Vert
y\right\Vert =1$ and $\left\Vert x-y\right\Vert \geq\varepsilon$ imply that
$\left\Vert \left(  x+y\right)  /2\right\Vert \leq1-\delta(\varepsilon)$.
\end{definition}

It follows from the Parallelogram Identity that every Hilbert space is
uniformly convex. Next we present two known theorems, the
Krasnosel'ski\u{\i}-Mann-Opial theorem \cite{Krasnoselskii, mann, Opial67} and
the Halpern-Suzuki theorem \cite{halpern, Suzuki}.

\begin{theorem}
\label{Th:KMO}\cite{Krasnoselskii, mann, Opial67} Let $\mathcal{H}$ be a real
Hilbert space and $D\subset\mathcal{H}$ be a nonempty, closed and convex
subset of $\mathcal{H}$. Given an averaged operator $h:D\rightarrow D$ with
$\operatorname*{Fix}(h)\neq\emptyset$ and an arbitrary $x^{0}\in D$, the
sequence generated by the recursion $x^{k+1}=h(x^{k}),$ $k\geq0$, converges
weakly to a point $z\in\operatorname*{Fix}(h)$.
\end{theorem}

\begin{theorem}
\label{Th:Halpern}\cite{halpern, Suzuki} Let $\mathcal{H}$ be a real Hilbert
space and $D\subset\mathcal{H}$ be a closed and convex subset of $\mathcal{H}%
$. Given an averaged operator $h:D\rightarrow D,$ and a sequence $\{\alpha
_{k}\}_{k=0}^{\infty}\subset\lbrack0,1]$ satisfying $\lim_{k\rightarrow\infty
}\alpha_{k}=0$ and $\sum\limits_{k=0}^{\infty}\alpha_{k}=\infty$, the sequence
$\{x^{k}\}_{k=0}^{\infty}$ generated by $x^{0}\in D$ and $x^{k+1}=\alpha
_{k}x^{0}+(1-\alpha_{k})h(x^{k}),$ $k\geq0$, converges strongly to a point
$z\in\operatorname*{Fix}(h)$.
\end{theorem}

\section{Weak convergence\label{Sec: Weak}}

In this section we first present an algorithm for solving the SCNPP$(p,r)$ for
two set-valued maximal monotone mappings. Then, for the general case of more
than two such set-valued mappings, we employ a product space formulation in
order to transform it into an SCNPP$(1,1)$ for two set-valued maximal monotone
mappings, in a similar fashion to what has been done in \cite[Section 4]{CS09}
and \cite[Subsection 6.1]{cgr}.

\subsection{The SCNPP$(1,1)$ for set-valued maximal monotone mappings}

Consider the SCNPP$(p,r)$ (\ref{eq:mszp-1})--(\ref{eq:mszp-2}) with $p=r=1$.
That is, given two set-valued mappings $B_{1}:\mathcal{H}_{1}\rightarrow
2^{\mathcal{H}_{1}}$ and $F_{1}:\mathcal{H}_{2}\rightarrow2^{\mathcal{H}_{2}%
},$ and a bounded linear operator $A:\mathcal{H}_{1}\rightarrow\mathcal{H}%
_{2}$, we want to%
\begin{equation}
\text{find a point }x^{\ast}\in\mathcal{H}_{1}\text{ such that }0\in
B_{1}(x^{\ast})\text{ and }0\in F_{1}(A\left(  x^{\ast}\right)  ).
\label{eq:two-SZP}%
\end{equation}
Here is our proposed algorithm for solving (\ref{eq:two-SZP}).

\begin{algorithm}
\label{alg:CGR}$\left.  {}\right.  $

\textbf{Initialization:} Let $\lambda>0$ and select an arbitrary starting
point $x^{0}\in\mathcal{H}_{1}$.

\textbf{Iterative step:} Given the current iterate $x^{k},$ compute%
\begin{equation}
x^{k+1}=J_{\lambda}^{B_{1}}\left(  x^{k}-\gamma A^{\ast}(I-J_{\lambda}^{F_{1}%
})A\left(  x^{k}\right)  \right)  , \label{eq:cgr-iterate}%
\end{equation}
where $A^{\ast}$ is the adjoint of $A$, $L=\left\Vert A^{\ast}A\right\Vert $
and $\gamma\in(0,2/L)$.
\end{algorithm}

The convergence theorem for this algorithm is presented next. We denote by
$\Gamma$ the solution set of (\ref{eq:two-SZP}).

\begin{theorem}
\label{Theorem:1}Let $\mathcal{H}_{1}$ and $\mathcal{H}_{2}$ be two real
Hilbert spaces. Given two set-valued maximal monotone mappings $B_{1}%
:\mathcal{H}_{1}\rightarrow2^{\mathcal{H}_{1}}$ and $F_{1}:\mathcal{H}%
_{2}\rightarrow2^{\mathcal{H}_{2}},$ and a bounded linear operator
$A:\mathcal{H}_{1}\rightarrow\mathcal{H}_{2},$ any sequence $\left\{
x^{k}\right\}  _{k=0}^{\infty}$ generated by Algorithm \ref{alg:CGR} converges
weakly to a point $x^{\ast}\in\Gamma$, provided that $\Gamma\neq\emptyset$ and
$\gamma\in(0,2/L)$, where $L=\left\Vert A^{\ast}A\right\Vert $.
\end{theorem}

\begin{proof}
In view of the connection between our SCNPP$(p,r)$ and Moudafi's SMVI, this
theorem can be obtained as a corollary of \cite[Theorem 3.1]{Moudafi}, the
proof of which is based on the Krasnosel'ski\u{\i}-Mann-Opial theorem
\cite{Krasnoselskii, mann, Opial67}.
\end{proof}

\begin{remark}
Observe that in Theorem \ref{Theorem:1} we assume that $\gamma\in(0,2/L)$,
while in \cite[Theorem 6.3]{cgr}, $\gamma$ is assumed to be in $(0,1/L)$,
which obviously was a more restrictive assumption.
\end{remark}

To describe the relationship of our work with splitting methods, let
$\mathcal{H}$ be a real Hilbert space, and let $B:\mathcal{H}\rightarrow
2^{\mathcal{H}}$ and $F:\mathcal{H}\rightarrow2^{\mathcal{H}}$ be two maximal
monotone mappings. Consider the following problem:%
\begin{equation}
\text{find a point }x^{\ast}\in\mathcal{H}\text{ such that }0\in B(x^{\ast
})+F(x^{\ast}). \label{eq:sum-zero}%
\end{equation}
Many algorithms were developed for solving this problem. An important class of
such algorithms is the class of splitting methods. References on splitting
methods and their applications can be found in Eckstein's Ph.D. thesis
\cite{Eckstein89}, in Tseng's work \cite{Tseng90, Tseng91, Tseng00} and more
recently in Combettes et al. \cite{Combettes04, cw05, cp11}.

One splitting method of interest is the following forward-backward algorithm:%
\begin{equation}
x^{k+1}=J^{B}\left(  I-h \right)  \left(  x^{k}\right)  , \label{eq:fb}%
\end{equation}
where $F = h$ is single-valued. Combettes \cite[Section 6]{Combettes04} was
interested in (\ref{eq:fb}) under the assumption that $B:\mathcal{H}%
\rightarrow2^{\mathcal{H}}$ and $h:\mathcal{H}\rightarrow\mathcal{H}$ are
maximal monotone, and $\beta h$ is firmly nonexpansive (i.e., $1/2$-av) for
some $\beta\in(0,\infty)$. He proposed the following algorithm:%
\begin{equation}
x^{k+1}=x^{k}+\lambda_{k}\left(  J_{\gamma_{k}}^{B}\left(  x^{k}-\gamma
_{k}(h\left(  x^{k}\right)  +b_{k})\right)  +a_{k}-x^{k}\right)  ,
\label{eq:combettes}%
\end{equation}
where the sequence $\left\{  \gamma_{k}\right\}  _{k=0}^{\infty}$ is bounded
and the sequences $\left\{  a_{k}\right\}  _{k=0}^{\infty}$ and $\left\{
b_{k}\right\}  _{k=0}^{\infty}$ are absolutely summable errors in the
computation of the resolvents. It can be seen that the iterative step
(\ref{eq:cgr-iterate}) is a special case of (\ref{eq:fb}) with $h=\gamma
A^{\ast}(I-J_{\lambda}^{F_{1}})A$. In the setting of Theorem \ref{Theorem:1}
here, $h$ is $1/\left(  \gamma L\right)  $-ism and therefore for
$\beta=(\gamma L)^{-1}$, the operator $\beta\gamma A^{\ast}(I-J_{\lambda
}^{F_{1}})A$ is $1$-ism, that is, firmly nonexpansive. Now by \cite[Example
20.27]{BC-book-2011}, this operator is maximal monotone. Therefore Algorithm
\ref{alg:CGR} is a special case of (\ref{eq:combettes}) without relaxation and
we also need to calculate the exact resolvent. It may be somewhat surprising
that our SCNPP is formulated in two different spaces, while (\ref{eq:sum-zero}%
) is only defined in one space and still we arrive at the same algorithm.
Further related results on proximal feasibility problems appear in Combettes
and Wajs \cite[Subsection 4.3]{cw05}.

\subsection{The \textbf{general} SCNPP$(p,r)$\label{Subsec: SCNPP-SCFPP}}

In view of Remark \ref{remark:resolvent}, we can show, by applying similar
arguments to those used in \cite{CS09}, that our SCNPP$(p,r)$ can be
transformed into a split common fixed point problem (SCFPP) (see
(\ref{eq:1.15})) with two operators\textbf{ }$T$\textbf{ }and\textbf{ }$U$ in
a product space. Next, we show how the general SCNPP$(p,r)$ can be transformed
into an SCNPP$(1,1)$ for two set-valued maximal monotone mappings.

Consider the space\textbf{ }$\boldsymbol{H}=\mathcal{H}_{1}^{p}\times
\mathcal{H}_{2}^{r},$\textbf{ }and the set-valued maximal monotone
mappings\textbf{ }$D:\mathcal{H}_{1}\rightarrow2^{\mathcal{H}_{1}}$ and
$\boldsymbol{F}:\boldsymbol{H}\rightarrow2^{\boldsymbol{H}}$\textbf{ }defined
by $D(x)=\{0\}$\textbf{ }for all\textbf{ }$x\in\mathcal{H}_{1}$\textbf{
}and\textbf{ }$\boldsymbol{F}(\left(  x^{1},\ldots,x^{p},y^{1},\ldots
,y^{r}\right)  )=B_{1}\left(  x^{1}\right)  \times\ldots\times B_{p}\left(
x^{p}\right)  \times F_{1}\left(  y^{1}\right)  \times\ldots\times
F_{r}\left(  y^{r}\right)  $ for each\textbf{ }$\left(  x^{1},\ldots
,x^{p},y^{1},\ldots,y^{r}\right)  \in\boldsymbol{H}.$\textbf{ }In addition,
let the bounded linear operator\textbf{ }$\boldsymbol{A}:\mathcal{H}%
_{1}\rightarrow\boldsymbol{H}$\textbf{ }be defined by\textbf{ }$\boldsymbol{A}%
\left(  x\right)  =\left(  x,\ldots,x,A_{1}\left(  x\right)  ,\ldots
,A_{r}\left(  x\right)  \right)  $\textbf{ }for all\textbf{ }$x\in
\mathcal{H}_{1}$. Then the general SCNPP$(p,r)$ (\ref{eq:mszp-1}%
)--(\ref{eq:mszp-2}) is equivalent to%
\begin{equation}
\text{find a point }x\in\mathcal{H}_{1}\text{ such that }0\in D(x)\text{ and
}0\in\boldsymbol{F}\left(  \boldsymbol{A}\left(  x\right)  \right)  .
\label{eq:prod}%
\end{equation}

When Algorithm \ref{alg:CGR} is applied to this two-set problem in the product
space $\boldsymbol{H}$ and then translated back to the original spaces, it
takes the following form.

\begin{algorithm}
\label{Alg: CCGR}$\left.  {}\right.  $

\textbf{Initialization:}$\ $Select an arbitrary starting point $x^{0}%
\in\mathcal{H}_{1}$.

\textbf{Iterative step: }Given the current iterate $x^{k},$ compute%
\begin{equation}
x^{k+1}=x^{k}+\gamma\left(  \sum_{i=1}^{p}\left(  J_{\lambda}^{B_{i}}%
(x^{k})-x^{k}\right)  +\sum_{j=1}^{r}A_{j}^{\ast}(J_{\lambda}^{F_{j}}%
-I)A_{j}\left(  x^{k}\right)  \right)  ,
\end{equation}
where $\gamma\in(0,2/L),$ with $L=p+\sum_{j=1}^{r}\Vert A_{j}\Vert^{2}$.
\end{algorithm}

The convergence of this algorithm follows from Theorem \ref{Theorem:1}. We may
also introduce relaxation parameters into the above algorithm as has been done
in the relaxed version of \cite[equation 2.10]{Moudafi10}.

\section{Strong convergence\label{Sec: Str}}

We focus on the SCNPP$(p,r)$ for two set-valued maximal monotone mappings,
keeping in mind that for the general case we can always apply the above
product space formulation and then translate back the algorithms to the
original spaces. In this section we first present a strong convergence theorem
for Algorithm \ref{alg:CGR} under an additional assumption. This result relies
on the work of Browder and Petryshyn \cite[Theorem 5]{Browder+Petryshyn}, and
on that of Baillon, Bruck and Reich \cite[Theorem 1.1]{bbr} (see also
\cite[Lemma 7]{Masad+Reich}). Then we study a second algorithm which is a
modification of Algorithm \ref{alg:CGR} that results in a Halpern-type
algorithm. The third algorithm in this section is inspired by Haugazeau's
method \cite{Haugazeau68}; see also \cite{BC01}.

\subsection{Strong convergence of Algorithm \ref{alg:CGR}}

The next two theorems are needed for our proof of Theorem
\ref{Theorem:Strong*}. We present their full proofs for the reader's convenience.

\begin{theorem}
\label{Theorem:browder}\cite[Theorem 5]{Browder+Petryshyn}, \cite{Ishikawa}
Let $\mathcal{B}$ be a uniformly convex Banach space. If the operator
$S:\mathcal{B}\rightarrow\mathcal{B}$ is nonexpansive with a nonempty fixed
point set $\operatorname*{Fix}\left(  S\right)  \neq\emptyset$, then for any
given constant $c\in(0,1)$, the operator $S_{c}:=cI+(1-c)S$ is asymptotically
regular and has the same fixed points as $S$.
\end{theorem}

\begin{proof}
It is obvious that $\operatorname*{Fix}\left(  S\right)  =\operatorname*{Fix}%
\left(  S_{c}\right)  $ and that $S_{c}$ is also a nonexpansive self-mapping
of $\mathcal{B}$. Let $u\in\operatorname*{Fix}\left(  S_{c}\right)  $ and for
a given $x\in\mathcal{B}$, let $x^{k}=S_{c}^{k}(x).$ Since $S_{c}$ is
nonexpansive and $u\in\operatorname*{Fix}\left(  S_{c}\right)  ,$ it follows
that%
\begin{equation}
\left\Vert x^{k+1}-u\right\Vert \leq\left\Vert x^{k}-u\right\Vert \text{ for
all }k\geq0.
\end{equation}
Therefore there exists $\lim_{k\rightarrow\infty}\left\Vert x^{k}-u\right\Vert
=\ell\geq0$. Assume that $\ell>0$. Then%
\begin{align}
x^{k+1}-u  &  =S_{c}^{k+1}(x)-u=S_{c}(x^{k})-u\nonumber\\
&  =\left(  cI+(1-c)S\right)  (x^{k})-u\nonumber\\
&  =c(x^{k}-u)+(1-c)\left(  S(x^{k})-u\right)  .
\end{align}
Since%
\begin{equation}
\lim_{k\rightarrow\infty}\left\Vert x^{k}-u\right\Vert =\lim_{k\rightarrow
\infty}\left\Vert x^{k+1}-u\right\Vert =\ell
\end{equation}
and%
\begin{equation}
\left\Vert x^{k+1}-u\right\Vert =\left\Vert S(x^{k})-u\right\Vert
\leq\left\Vert x^{k}-u\right\Vert ,
\end{equation}
the uniform convexity of $\mathcal{B}$ implies that%
\begin{equation}
\lim_{k\rightarrow\infty}\left\Vert \left(  x^{k}-u\right)  -\left(
S(x^{k})-u\right)  \right\Vert =0,
\end{equation}
i.e., $x^{k}-S(x^{k})\rightarrow0.$ Hence $x^{k+1}-x^{k}\rightarrow0$, which
means that $S_{c}$ is asymptotically regular, as claimed.
\end{proof}

\begin{theorem}
\label{Theorem:bbr}\cite[Theorem 1.1]{bbr} Let $\mathcal{B}$ be a uniformly
convex Banach space. If the operator $S:\mathcal{B}\rightarrow\mathcal{B}$ is
nonexpansive, odd and asymptotically regular at $x\in\mathcal{B}$, then the
sequence $\left\{  S^{k}(x)\right\}  _{k=0}^{\infty}$ converges strongly to a
fixed point of $S$.
\end{theorem}

\begin{proof}
Since $S$ is odd, $S(0)=-S(0)$ and $S(0)=0$. Since $S$ is nonexpansive, we
have by the triangle inequality,%
\begin{align}
\left\Vert S^{k}(x)\right\Vert  &  =\left\Vert S^{k}(x)\right\Vert -\left\Vert
S^{k}(0)\right\Vert \leq\left\Vert S^{k}(x)-S^{k}(0)\right\Vert \nonumber\\
&  \leq\left\Vert S^{k-1}(x)-S^{k-1}(0)\right\Vert =\left\Vert S^{k-1}%
(x)\right\Vert \leq\cdots\leq\left\Vert x-0\right\Vert =\left\Vert
x\right\Vert ,
\end{align}
which means that the sequence $\left\{  \left\Vert S^{k}(x)\right\Vert
\right\}  _{k=0}^{\infty}$ is decreasing and bounded. Therefore the limit
$\lim_{k\rightarrow\infty}\left\Vert S^{k}(x)\right\Vert $ exists and, for a
fixed $i,$ the sequence\newline$\left\{  \left\Vert S^{k+i}(x)+S^{k}%
(x)\right\Vert \right\}  _{k=0}^{\infty}$ is decreasing. Let $\lim
_{k\rightarrow\infty}\left\Vert S^{k}(x)\right\Vert =d$. Then by the triangle
inequality,%
\begin{align}
2d  &  \leq\left\Vert 2S^{k}(x)\right\Vert =\left\Vert S^{k}(x)-S^{k+i}%
(x)+S^{k+i}(x)+S^{k}(x)\right\Vert \nonumber\\
&  \leq\left\Vert S^{k}(x)-S^{k+i}(x)\right\Vert +\left\Vert S^{k}%
(x)+S^{k+i}(x)\right\Vert .
\end{align}
Since $S$ is asymptotically regular at $x$, $\lim_{k\rightarrow\infty
}\left\Vert S^{k}(x)-S^{k+i}(x)\right\Vert =0$. Thus $\lim_{k\rightarrow
\infty}\left\Vert S^{k}(x)+S^{k+i}(x)\right\Vert \geq2d.$ But the sequence
$\left\{  \left\Vert S^{k+i}(x)+S^{k}(x)\right\Vert \right\}  _{k=0}^{\infty}$
is decreasing, so that $\left\Vert S^{k}(x)+S^{k+i}(x)\right\Vert \geq2d$ for
all $k$ and $i$. We now have $\lim_{k\rightarrow\infty}\left\Vert
S^{k}(x)\right\Vert =d$ and $\lim_{m,n\rightarrow\infty}\left\Vert
S^{n}(x)+S^{m}(x)\right\Vert =2d$. The uniform convexity of $\mathcal{B}$
implies that $\lim_{m,n\rightarrow\infty}\left\Vert S^{n}(x)-S^{m}%
(x)\right\Vert =0$, whence $\left\{  S^{k}(x)\right\}  _{k=0}^{\infty}$
converges strongly to a fixed point of $S.$
\end{proof}

In Theorem \ref{Theorem:Strong*} we need the resolvent $J_{\lambda}^{B}$ to be
odd, which means that%
\begin{equation}
\left(  \left(  I+\lambda B\right)  ^{-1}\right)  (-x)=-\left(  \left(
I+\lambda B\right)  ^{-1}\right)  (x)\text{\ for all }x\in\mathcal{H}.
\label{eq:odd}%
\end{equation}
Denote%
\begin{equation}
\left(  \left(  I+\lambda B\right)  ^{-1}\right)  (-x)=y\text{ and }\left(
\left(  I+\lambda B\right)  ^{-1}\right)  (x)=z.
\end{equation}
Then%
\begin{equation}
-x\in y+\lambda B(y)\text{ and }x\in z+\lambda B(z).
\end{equation}
If $B$ is odd, then%
\begin{equation}
x\in-y+\lambda B(-y).
\end{equation}
Hence $-y=z$, which is (\ref{eq:odd}). Therefore we assume in the following
theorem that both $B_{1}$ and $F_{1}$ are odd.

Now we are ready to present the strong convergence theorem for Algorithm
\ref{alg:CGR}. Its proof relies on Theorem \ref{Theorem:bbr}.

\begin{theorem}
\label{Theorem:Strong*} Let $\mathcal{H}_{1}$ and $\mathcal{H}_{2}$ be two
real Hilbert spaces. Let two set-valued, odd and maximal monotone mappings
$B_{1}:\mathcal{H}_{1}\rightarrow2^{\mathcal{H}_{1}}$ and $F_{1}%
:\mathcal{H}_{2}\rightarrow2^{\mathcal{H}_{2}},$ and a bounded linear operator
$A:\mathcal{H}_{1}\rightarrow\mathcal{H}_{2}$ be given. If $\gamma\in(0,2/L),$
then any sequence $\left\{  x^{k}\right\}  _{k=0}^{\infty}$ generated by
Algorithm \ref{alg:CGR} converges strongly to $x^{\ast}\in\Gamma$.
\end{theorem}

\begin{proof}
The operator $J_{\lambda}^{B_{1}}\left(  I-\gamma A^{\ast}(I-J_{\lambda
}^{F_{1}})A\right)  $ is averaged by the proof of \cite[Theorem 3.1]{Moudafi}.
Therefore, by \cite[Theorem 5]{Browder+Petryshyn} and \cite{Ishikawa} (see
Theorem \ref{Theorem:browder}), the operator $J_{\lambda}^{B_{1}}\left(
I-\gamma A^{\ast}(I-J_{\lambda}^{F_{1}})A\right)  $ is also asymptotically
regular. Since $B_{1}$ and $F_{1}$ are odd, so are their resolvents
$J_{\lambda}^{B_{1}}$ and $J_{\lambda}^{F_{1}},$ and therefore $J_{\lambda
}^{B_{1}}\left(  I-\gamma A^{\ast}(I-J_{\lambda}^{F_{1}})A\right)  $ is odd.
Finally, the strong convergence of Algorithm \ref{alg:CGR} is now seen to
follow from Theorem \ref{Theorem:bbr}.
\end{proof}

For the general SCNPP$(p,r)$ we can again employ a product space formulation
as in Subsection \ref{Subsec: SCNPP-SCFPP} and under the additional oddness
assumption also get strong convergence.

\subsection{A Halpern-type algorithm}

Next, we consider a modification of Algorithm \ref{alg:CGR} inspired by
Halpern's iterative method and prove its strong convergence. Let
$T:C\rightarrow C$ be a nonexpansive operator, where $C$ is a nonempty, closed
and convex subset of a Banach space $\mathcal{B}$. A classical way to study
nonexpansive mappings is to use strict contractions to approximate $T$, i.e.,
for $t\in(0,1)$, we define the strict contraction $T_{t}:C\rightarrow C$ by%
\begin{equation}
T_{t}(x)=tu+(1-t)T(x)\text{ for }x\in C,
\end{equation}
where $u\in C$ is fixed. Banach's Contraction Mapping Principle (see, e.g.,
\cite{Goebel}) guarantees that each $T_{t}$ has a unique fixed point $x_{t}\in
C$. In case $\operatorname*{Fix}(T)\neq\emptyset$, Browder \cite{Browder}
proved that if $\mathcal{B}$ is a Hilbert space, then $x_{t}$ converges
strongly as $t\rightarrow0^{+}$ to the fixed point of $T$ nearest to $u$.
Motivated by Browder's results, Halpern \cite{halpern} proposed an explicit
iterative scheme and proved its strong convergence to a point $z\in
\operatorname*{Fix}(T)$. In the last decades many authors modified Halpern's
iterative scheme and found necessary and sufficient conditions, concerning the
control sequence, that guarantee the strong convergence of Halpern-type
schemes (see, e.g., \cite{lions, Reich, Wittmann, Xu, Suzuki}). Our algorithm
for the SCNPP$(p,r)$ with two set-valued maximal monotone mappings is
presented next.

\begin{algorithm}
\label{alg:CGR-Halpern}$\left.  {}\right.  $

\textbf{Initialization:} Select some $\lambda>0$ and an arbitrary starting
point $x^{0}\in\mathcal{H}_{1}$.

\textbf{Iterative step:} Given the current iterate $x^{k},$ compute%
\begin{equation}
x^{k+1}=\alpha_{k}x^{0}+(1-\alpha_{k})J_{\lambda}^{B_{1}}\left(  I-\gamma
A^{\ast}(I-J_{\lambda}^{F_{1}})A\right)  \left(  x^{k}\right)  ,
\end{equation}
where $\gamma\in(0,2/L)$ with $L=\left\Vert A^{\ast}A\right\Vert $ and the
sequence $\{\alpha_{k}\}_{k=0}^{\infty}\subset\lbrack0,1]$ satisfies
$\lim_{k\rightarrow\infty}\alpha_{k}=0$ and $\sum\limits_{k=0}^{\infty}%
\alpha_{k}=\infty$.
\end{algorithm}

Here is our strong convergence theorem for this algorithm.

\begin{theorem}
\label{Theorem:Halpern}Let $\mathcal{H}_{1}$ and $\mathcal{H}_{2}$ be two real
Hilbert spaces. Let there be given two set-valued maximal monotone mappings
$B_{1}:\mathcal{H}_{1}\rightarrow2^{\mathcal{H}_{1}}$ and $F_{1}%
:\mathcal{H}_{2}\rightarrow2^{\mathcal{H}_{2}},$ and a bounded linear operator
$A:\mathcal{H}_{1}\rightarrow\mathcal{H}_{2}$. If $\Gamma\neq\emptyset$,
$\gamma\in(0,2/L)$ and $\{\alpha_{k}\}_{k=0}^{\infty}\subset\lbrack0,1]$
satisfies $\lim_{k\rightarrow\infty}\alpha_{k}=0$ and $\sum\limits_{k=0}%
^{\infty}\alpha_{k}=\infty$, then any sequence $\left\{  x^{k}\right\}
_{k=0}^{\infty}$ generated by Algorithm \ref{alg:CGR-Halpern} converges
strongly to $x^{\ast}\in\Gamma$.
\end{theorem}

\begin{proof}
As we already know, the operator $J_{\lambda}^{B_{1}}\left(  I-\gamma A^{\ast
}(I-J_{\lambda}^{F_{1}})A\right)  $ is averaged. So, according to Theorem
\ref{Th:Halpern}, any sequence $\left\{  x^{k}\right\}  _{k=0}^{\infty}$
generated by Algorithm \ref{alg:CGR-Halpern} converges strongly to a point in
the fixed point set of the operator, i.e., $x^{\ast}\in\operatorname*{Fix}%
\left(  J_{\lambda}^{B_{1}}\left(  I-\gamma A^{\ast}(I-J_{\lambda}^{F_{1}%
})A\right)  \right)  $ as long as this set is nonempty. As in the proof of
\cite[Theorem 3.1]{Moudafi}, we conclude that $x^{\ast}\in\Gamma,$ as claimed.
\end{proof}

\subsection{An \textbf{Haugazeau-type algorithm}}

Haugazeau \cite{Haugazeau68} presented an algorithm for solving the
\textit{Best Approximation Problem} (BAP) of finding the projection of a point
onto the intersection of $m$ closed convex subsets\textbf{ }$\{C_{i}%
\}_{i=1}^{m}\subset\mathcal{H}$ of a real Hilbert space. Defining for any pair
$x,y\in\mathcal{H}$\textbf{ }the set%
\begin{equation}
H(x,y):=\{u\in\mathcal{H}\mid\left\langle u-y,x-y\right\rangle \leq0\},
\end{equation}
and denoting by $T(x,y,z)$ the projection of $x$\ onto $H(x,y)\cap H(y,z)$,
namely, $T(x,y,z)=P_{H(x,y)\cap H(y,z)}(x),$ Haugazeau showed that for an
arbitrary starting point\textbf{ }$x^{0}\in\mathcal{H}$, any sequence
$\{x^{k}\}_{k=0}^{\infty},$ generated by the iterative step%
\begin{equation}
x^{k+1}=T(x^{0},x^{k},P_{k(\operatorname*{mod}m)+1}(x^{k}))
\end{equation}
converges strongly to the projection of $x^{0}$ onto $C=\cap_{i=1}^{m}C_{i}$.
The operator $T$ requires projecting onto the intersection of two
constructible half-spaces; this is not difficult to implement. In
\cite{Haugazeau68} Haugazeau introduced the operator $T$ as an explicit
description of the projector onto the intersection of the two half-spaces
$H(x,y)$ and $H(y,z)$. So, following, e.g., \cite[Definition 3.1]{bcl06}, and
denoting $\pi=\left\langle x-y,y-z\right\rangle ,$ $\mu=\Vert x-y\Vert^{2},$
$\nu=\Vert y-z\Vert^{2}$ and $\rho=$ $\mu\nu-\pi^{2}$, we have%
\begin{equation}
T(x,y,z)=\left\{
\begin{array}
[c]{ll}%
z, & \text{if }\rho=0\text{ and }\pi\geq0,\medskip\\
x+\left(  1+\frac{\pi}{\nu}\right)  (z-y), & \text{if }\rho>0\text{ and }%
\pi\nu\geq\rho,\medskip\\
y+\frac{\nu}{\rho}(\pi(x-y)+\mu(z-y)), & \text{if }\rho>0\text{ and }\pi
\nu<\rho.\medskip
\end{array}
\right.  \label{eq:Q(x,y,z)}%
\end{equation}
We already know that the operator $S:=J_{\lambda}^{B_{1}}\left(  I-\gamma
A^{\ast}(I-J_{\lambda}^{F_{1}})A\right)  $ is averaged and therefore
nonexpansive. Now consider the firmly nonexpansive operator $S_{1/2}:=\left(
I+S\right)  /2$, which according to Theorem \ref{Theorem:browder} has the same
fixed points as $S$. Following the \textquotedblleft weak-to-strong
convergence principle\textquotedblright\ \cite{BC01}, strong convergence
(without additional assumptions) can be obtained by replacing the updating
rule (\ref{eq:cgr-iterate})\ in Algorithm \ref{alg:CGR} with%
\begin{align}
x^{k+1}  &  =T\left(  x^{0},x^{k},S_{1/2}\left(  x^{k}\right)  \right)
\nonumber\\
&  =P_{H\left(  x^{0},x^{k}\right)  \cap H\left(  x^{k},S_{1/2}\left(
x^{k}\right)  \right)  }(x^{0}).
\end{align}
A similar technique can also be applied to the forward-backward splitting
method in \cite[Section 6]{Combettes04}.

\section{Further comments\label{sec:comments}}

\begin{enumerate}
\item Since the SCNPP$(p,r)$ generalizes the SVIP, it includes all the
applications to which SVIP applies (see \cite[Section 7]{cgr}). In particular,
it includes the Split Feasibility Problem (SFP) and the Convex Feasibility
Problem\textit{ }(CFP). Since the Common Solutions to Variational Inequalities
Problem (CSVIP) \cite{cgrs} with operators is a special case of the SVIP, the
SCNPP$(p,r)$ includes its applications as well. In addition, since all the
applications of the SMVI presented in \cite{Moudafi} are for $f=g=0$ in
(\ref{eq:mvi-1}) and (\ref{eq:mvi-2}) above, it follows that these
applications are also covered by our SCNPP$(p,r)$. They include the Split
Minimization Problem (SMP), which has already been presented in
\cite[Subsection 7.3]{cgr} with continuously differentiable convex functions,
for which we can now drop this assumption, the Split Saddle-Point Problem
(SSPP), the Split Minimax Problem (SMMP) and the Split Equilibrium Problem
(SEP). Observe that if $\mathcal{H}_{1}=\mathcal{H}_{2}$ and $A_{j}=I$ for for
all $j=1,2,\ldots,r,$ then we can deal with all of the above applications with
\textquotedblleft Split\textquotedblright\ replaced by \textquotedblleft
Common\textquotedblright.\ We can even study mixtures of \textquotedblleft
split\textquotedblright\ and \textquotedblleft common\textquotedblright\ applications.

\item According to Remark \ref{remark:resolvent-2}, the operator $J_{\lambda
}^{B}(I-\lambda f)$ is averaged, where $B:\mathcal{H}\rightarrow
2^{\mathcal{H}}$ is maximal monotone, the operator $f:\mathcal{H\rightarrow
H}$ is $\alpha$-ism and $\lambda\in(0,2\alpha)$. Since our convergence
theorems rely on the averagedness of the operators involved, we could modify
our algorithms and obtain strong convergence for Moudafi's SMVI
((\ref{eq:mvi-1}) and (\ref{eq:mvi-2}) above). In addition, our algorithms
allow us to solve Moudafi's SMVI with monotone and hemicontinuous operators
$f$ and $g$ (which is a larger class than the class of inverse strongly
monotone operators).

\item Assuming that the set-valued mappings $B_{1}:\mathcal{H}_{1}%
\rightarrow2^{\mathcal{H}_{1}}$ and $B_{2}:\mathcal{H}_{2}\rightarrow
2^{\mathcal{H}_{2}}$ are maximal monotone, and $f:\mathcal{H}_{1}%
\rightarrow\mathcal{H}_{1}$ and $g:\mathcal{H}_{2}\rightarrow\mathcal{H}_{2}$
are ism-operators, Moudafi presented an algorithm that converges weakly to a
solution of the SMVI. By \cite[Theorem 3]{Rockafellar}, the sum of a maximal
monotone mapping and an ism-operator is maximal monotone. Therefore, the SMVI
reduces to our set-valued two-mapping SCNPP$(p,r)$. In addition, we can phrase
the set-valued SVIP for maximal monotone mappings in the following way. Given
two maximal monotone mappings $B_{1}:\mathcal{H}_{1}\rightarrow2^{\mathcal{H}%
_{1}}$ and $B_{2}:\mathcal{H}_{2}\rightarrow2^{\mathcal{H}_{2}}$, a bounded
linear operator $A:\mathcal{H}_{1}\rightarrow\mathcal{H}_{2}$, and nonempty,
closed and convex subsets $C\subset\mathcal{H}_{1}$ and $Q\subset
\mathcal{H}_{2}$, the set-valued SVIP is formulated as follows:%
\begin{gather}
\text{find a point }x^{\ast}\in C\text{ and a point }u^{\ast}\in B_{1}%
(x^{\ast})\nonumber\\
\text{such that }\left\langle u^{\ast},x-x^{\ast}\right\rangle \geq0\text{ for
all }x\in C,\nonumber\\
\text{and such that}\nonumber\\
\text{the points }y^{\ast}=A\left(  x^{\ast}\right)  \in Q\text{ and }v^{\ast
}\in B_{2}(y^{\ast})\nonumber\\
\text{solve }\left\langle v^{\ast},y-y^{\ast}\right\rangle \geq0\text{ for all
}y\in Q.
\end{gather}
It is clear that if the zeros of the set-valued mappings $B_{1}$ and $B_{2}$
are in $C$ and $Q$, respectively, then they are solutions of the set-valued
SVIP, but in general not all solutions are zeros.
\end{enumerate}

\bigskip

\textbf{Acknowledgments}

This work was partially supported by United States-Israel Binational Science
Foundation (BSF) Grant number 200912, US Department of Army Award number
W81XWH-10-1-0170, Israel Science Foundation (ISF) Grant number 647/07, the
Fund for the Promotion of Research at the Technion, and by the Technion VPR
Fund.

\end{document}